\documentclass[11pt]{article}

\usepackage{amsthm, amsmath, amssymb, amsfonts, url, booktabs, tikz, setspace, fancyhdr, bm}
\usepackage{geometry}
\usepackage{hyperref, enumerate}
\hypersetup{hidelinks}
\usepackage[shortlabels]{enumitem}
\usepackage[babel]{microtype}
\usepackage[english]{babel}
\usepackage[capitalise]{cleveref}
\usepackage{comment}
\usepackage{bbm}
\usepackage{csquotes}
\usepackage{mathabx}
\usepackage{graphicx}
\usepackage{float}
\usepackage{xcolor}
\usepackage{mathtools}
\usepackage{mathrsfs}
\usepackage{authblk}

\usetikzlibrary{positioning, arrows.meta, shapes.geometric}
\usetikzlibrary{decorations.pathmorphing}

\counterwithin{figure}{section}

\allowdisplaybreaks

\newtheorem{theorem}{Theorem}[section]

\newtheorem{proposition}[theorem]{Proposition}
\newtheorem{lemma}[theorem]{Lemma}

\theoremstyle{definition}

\newtheorem*{defn-non}{Definition}

\newtheorem{remark}[theorem]{Remark}

\newlist{Case}{enumerate}{2}
\setlist[Case,1]{%
	label={\bfseries Case \arabic*.},
	labelindent=1em,
	labelwidth=1.3cm,
	labelsep*=1em,
	leftmargin=!
}
\setlist[Case,2]{%
	label={\bfseries Subcase \arabic{Casei}.\arabic*.},
	labelindent=-1em,
	labelwidth=1.3cm,
	labelsep*=1em,
	leftmargin=!
}

\DeclareMathOperator{\Gen}{Gen}
\DeclareMathOperator{\ord}{ord}

\newcommand{\C}{\mathbb C}

\newcommand{\cB}{\mathcal B}
\newcommand{\cL}{\mathcal L}

\title{When Do Subset Sums in Finite Abelian Groups Support $2$-Designs?}

\author{
	Hengfeng Liu%
	\thanks{\raggedright School of Mathematics, Southwest Jiaotong University, Chengdu, China.
		\mbox{E-mail: hengfengliu@163.com.}},
	\quad
	Chunming Tang%
	\thanks{\raggedright School of Information Science and Technology,
		Southwest Jiaotong University, Chengdu, China. 
		\mbox{E-mail: tangchunmingmath@163.com.}},
	\quad
	Cuiling Fan%
	\thanks{\raggedright School of Mathematics, Southwest Jiaotong University, Chengdu, China.
		\mbox{E-mail: cuilingfan@163.com.}},
	\quad
	Zhengchun Zhou%
	\thanks{\raggedright School of Information Science and Technology,
		Southwest Jiaotong University, Chengdu, China. \mbox{E-mail: zzc@swjtu.edu.cn.}}
}

\date{}

\begin{document}
\maketitle

\begin{abstract}
	Subset sums over finite abelian groups lie at the intersection of
	additive combinatorics, design theory, and coding theory. Let $G$ be a finite abelian group, and let $\cB_k^x$ be the family of $k$-subsets of $G$ whose elements sum to $x\in G$. This paper studies when the incidence structure $(G,\cB_k^x)$ is a block design. The elementary abelian $p$-group case was settled by Falcone and Pavone. Pavone (\emph{Des. Codes Cryptogr.} 91 (2023), 2585--2603) further asked whether, for an arbitrary finite abelian group $G$, the zero-sum incidence structure $(G,\cB_k^0)$ can be a nontrivial $2$-design only when $G$ is an elementary abelian $p$-group. We settle this open question in the stronger form that, for every $x\in G$, $(G,\cB_k^x)$ can be a nontrivial $2$-design only if $G$ is an elementary abelian $p$-group. The proof develops a character-theoretic approach to subset-sum designs, using character sums over the blocks to constrain the structure of the character group $\widehat G$. The approach also yields a complete characterization of subset-sum $1$-designs and general arithmetic restrictions on subset-sum designs over arbitrary finite abelian groups, extending the corresponding results previously known for finite abelian $p$-groups.
\end{abstract}

\section{Introduction}

Let $G$ be an abelian group, written additively, and let $U$ be a finite subset of $G$. Denote by $N(U,k,x)$ the number of $k$-subsets of $U$ that sum to $x\in G$:
\[
N(U,k,x)
=
\#\left\{
S\subseteq U:\ |S|=k,\ \sum_{g\in S}g=x
\right\}.
\]
The classical \emph{subset-sum problem} asks whether $N(U,k,x)>0$. This NP-complete problem \cite{Karp1972} arises naturally in additive combinatorics and has important applications in coding theory and cryptography. Typical examples include knapsack cryptosystems for $G=\mathbb Z$ \cite{MerkleHellman1978}, the deep-hole problem for extended Reed--Solomon codes \cite{Cheng2008,ZhangWanKaipa2020,ZhangWan2023}, and the minimum-distance problem for elliptic curve codes \cite{LiWanZhang2015,ZhangWan2023}. More recently, it has also appeared in computational lattice theory, where it was used to prove that the shortest vector problem in Euclidean space is NP-hard \cite{Wan2026}. The main difficulty of the subset-sum problem lies in the freedom to choose the subset $U$. Nevertheless, considerable progress has been made when $U=G$. Li and Wan \cite{LiWan2008} obtained a closed formula for $N(G,k,x)$ when $G$ is the additive group of a finite field. They subsequently extended the result to arbitrary finite abelian groups \cite{LiWan2012}, and Kosters \cite{Kosters2013} later gave a short proof using character theory.

A $t$-$(n,k,\lambda)$ \emph{block design} is an incidence structure $(\mathcal P,\mathcal B)$ in which $\mathcal P$ is a set of $n$ points and $\mathcal B$ is a family of $k$-subsets of $\mathcal P$, called blocks, such that every $t$-subset of $\mathcal P$ is contained in exactly $\lambda$ blocks. Throughout this paper, we say a $t$-$(n,k,\lambda)$ design is nontrivial if $t<k<n$ and $\lambda>0$.  From the viewpoint of design theory, subset-sum families give natural incidence structures on the underlying group. The zero-sum condition already occurs in standard geometric examples, including the point-flat designs of the affine geometry \(\mathrm{AG}(d,p)\) and the projective geometry \(\mathrm{PG}(d,2)\). Motivated by such examples, Caggegi et al.~\cite{CaggegiFalconePavone2017} introduced additive designs, a notion that has recently attracted considerable attention \cite{BurattiGaliciMontinaroNakicWassermann2026,BurattiMerolaNakic2025}. A block design \((\mathcal P,\mathcal B)\) is called \emph{additive} if the point set \(\mathcal P\) can be embedded in an abelian group \(G\) so that the sum of the points in every block is zero. It is called \emph{strongly additive} if, under such an embedding, the blocks are precisely all zero-sum \(k\)-subsets.
 
Conversely, given a finite abelian group \(G\), one is led to the following general problem: under what conditions is the incidence structure \((G,\cB_k^x)\) a block design?
The problem has been settled for elementary abelian $p$-groups. Falcone and Pavone \cite{FalconePavone2021} showed that the subset-sum $k$-subsets of $G=(\mathbb F_2^n,+)$ form a $2$-design if and only if $k$ is even; in that case they in fact form a $3$-design. For odd $p$, Pavone \cite{Pavone2023} proved that, for $G=(\mathbb F_p^n,+)$, \((G,\cB_k^x)\) is a $1$-design exactly when $p\mid k$, and a $2$-design exactly when $p\mid k$ and $x=0$. The same paper raised the following question in the zero-sum case: for an arbitrary finite abelian group $G$, can $(G,\cB_k^0)$ be a $2$-design only when $G$ is an elementary abelian $p$-group? More recently, the characterization of $1$-designs has been extended to arbitrary finite abelian $p$-groups \cite{LiuTangFanLuo2025}.
 
The problem is also closely connected with coding theory. For an elliptic curve $E$ over $\mathbb F_q$, subset-sum blocks in $E(\mathbb F_q)$ correspond to minimum-weight codewords in the associated elliptic curve codes \cite{Cheng2008,LiWanZhang2015,HanRen2024}. Exploiting this correspondence and choosing elliptic curves with $E(\mathbb F_q)\cong \mathbb Z_p\oplus\mathbb Z_p$, Liu et al.~\cite{LiuTangZhouHanChen2026} obtained the first generic construction of $q$-ary near-MDS codes supporting $2$-designs of length exceeding $q+1$.

In the present paper, we answer Pavone's question \cite{Pavone2023} in the following stronger form, using new character-theoretic techniques.
\begin{theorem}\label{thm:main}
	Let \(G\) be a finite abelian group. For any \(x\in G\) and \(k\) satisfying $2\le k\le |G|-1$, if \((G,\cB_k^x)\) is a \(2\)-\((|G|,k,\lambda)\) design $(\lambda>0)$, then $G$ is an elementary abelian \(p\)-group.
\end{theorem}

Since every \(t\)-design with \(t\ge2\) is a \(2\)-design, the theorem also excludes all nontrivial \(t\)-designs ($t\ge2$) arising from subset-sum families over non-elementary finite abelian groups. Combined with the known results on elementary abelian \(p\)-groups \cite{FalconePavone2021,Pavone2023}, this fully characterizes the condition for a finite abelian group to support subset-sum $2$-design.

The techniques developed here also yield a necessary and sufficient criterion for $(G,\cB_k^x)$ to be a $1$-design over an arbitrary finite abelian group $G$. This criterion specializes immediately to the characterization of $1$-designs over finite abelian $p$-groups in \cite[Theorem~3.3]{LiuTangFanLuo2025}. It further gives the arithmetic restrictions $\gcd(k,\exp(G))>1$ for $1$-designs and $\exp(G)\mid k$ for $2$-designs. These results are new for general finite abelian groups and extend those of Pavone \cite{Pavone2023} and Liu et al.~\cite{LiuTangFanLuo2025}.

\section{Preliminaries}\label{sec:preliminaries}
In this section, we recall the basic facts and auxiliary results used in the proof.

\subsection{Characters of finite abelian groups}
Let \(G\) be a finite abelian group of size $n$, written additively, and let \(\widehat G=\operatorname{Hom}(G,\mathbb C^\times)\) be its character group. For a standard reference on characters of finite abelian groups, see \cite{Serre1973}.

One of the most important properties of characters is orthogonality:
\begin{equation}\label{eq:character-orthogonality}
	\sum_{g\in G}\chi(g)=
	\begin{cases}
		n,&\chi=1,\\
		0,&\chi\ne 1,
	\end{cases}
	\qquad
	\sum_{\chi\in\widehat G}\chi(g)=
	\begin{cases}
		n,&g=0,\\
		0,&g\ne 0.
	\end{cases}
\end{equation}

For a subgroup $H\le\widehat G$, define its \emph{annihilator} by $H^\perp
=
\{g\in G:\chi(g)=1\text{ for every }\chi\in H\}$. Then $|H|\,|H^\perp|=n$ and
\begin{equation}\label{eq:subgroup-orthogonality}
	\sum_{\chi\in H}\chi(g)=
	\begin{cases}
		|H|,&g\in H^\perp,\\
		0,&g\notin H^\perp.
	\end{cases}
\end{equation}

For a function $f:G\to\mathbb C$, its \emph{Fourier transform} is given by
$$
\widehat f(\chi)
=
\sum_{g\in G}f(g)\chi(g),
\qquad
\chi\in\widehat G.
$$
The Fourier inversion formula is $f(g)
=
\frac1n
\sum_{\chi\in\widehat G}
\widehat f(\chi)\chi(-g)$. By orthogonality of characters, $f$ is constant if and only if
$\widehat f(\chi)=0$ for every nontrivial $\chi\in\widehat G$.

We shall use the following elementary identity, also noted in \cite{Kosters2013}.

\begin{lemma}
	\label{lem:product-identity}
	Let \(\chi\in\widehat G\) have order \(d\).  Then
	\[
	\prod_{g\in G}\bigl(1-\chi(g)Z\bigr)
	=
	\bigl(1-Z^d\bigr)^{n/d}.
	\]
\end{lemma}

\begin{proof}
	The image of $\chi$ is the group of $d$-th roots of unity, and every value
	has exactly $n/d$ preimages. Hence
	$$
	\prod_{g\in G}\bigl(1-\chi(g)Z\bigr)
	=
	\left(
	\prod_{\zeta^d=1}(1-\zeta Z)
	\right)^{n/d}
	=
	\bigl(1-Z^d\bigr)^{n/d}.
	$$
\end{proof}

We have the following representation of $N(U,k,x)$ by character sums. Throughout this paper, \(\operatorname{Coef}_{k}\left\{\frac{f(Z)}{g(Z)}\right\}\) denotes the $Z^k$-coefficient in the power series expansion $\frac{f(Z)}{g(Z)}$.

\begin{lemma}\label{lem:subset-sum-fourier}
	Let \(U\subseteq G\), \(1\le k\le n\), and \(x\in G\). Then
	\begin{equation}\label{eq:subset-sum-character-formula}
		N(U,k,x)=\frac1n\sum_{\chi\in\widehat G}\chi(-x)
		\operatorname{Coef}_{k}\left\{\prod_{g\in U}(1+\chi(g)Z)\right\}.
	\end{equation}
	In particular,
	\begin{equation}\label{eq:group-subset-sum-character-formula}
		N(G,k,x)=\frac1n\sum_{\chi\in\widehat G}
		\chi(-x)\operatorname{Coef}_{k}\left\{\bigl(1-(-Z)^{\ord(\chi)}\bigr)^{n/\ord(\chi)}\right\}.
	\end{equation}

\end{lemma}

\begin{proof}
	
	Expanding the product gives
	\begin{equation}\label{eq:subset-product-identity}
		\prod_{g\in U}(1+\chi(g)Z)
		=\sum_{S\subseteq U}\chi\left(\sum_{g\in S}g\right)Z^{|S|}.
	\end{equation}
	
	Using character orthogonality and Eq.~\eqref{eq:subset-product-identity}, we obtain Eq.~\eqref{eq:subset-sum-character-formula}. Equation~\eqref{eq:group-subset-sum-character-formula} then follows from Lemma~\ref{lem:product-identity}.
\end{proof}

\subsection{Primary decomposition and \texorpdfstring{$p$}{p}-torsion}

We shall use the canonical \emph{primary decomposition} of finite
abelian groups. Let $A$ be a finite abelian group, written
multiplicatively. For each prime $p\mid |A|$, define
\[
A_p
=
\left\{
a\in A:
a^{p^r}=1
\text{ for some } r\ge 0
\right\}.
\]
Then $A_p$ is the unique Sylow $p$-subgroup of $A$, also called the
\emph{$p$-primary component} of $A$. The primary decomposition of
$A$ is the internal direct product
\[
A=\prod_{p\mid |A|} A_p.
\]
In particular, every $a\in A$ admits a unique factorization $a=\prod_{p\mid |A|}a_p$, where $a_p\in A_p$.

For each prime $p\mid |A|$, there exist a unique positive integer $d_p$ and uniquely determined positive integers
$\lambda_{p,1}\ge\lambda_{p,2}\ge\cdots\ge\lambda_{p,d_p}\ge1$ such that
\[
A_p
\cong
\mathbb Z_{p^{\lambda_{p,1}}}
\oplus \mathbb Z_{p^{\lambda_{p,2}}}
\oplus\cdots\oplus
\mathbb Z_{p^{\lambda_{p,d_p}}}.
\]
Here $\mathbb Z_m$ denotes the cyclic group of order $m$.

For an integer $s\ge1$, the \emph{$p^s$-torsion} subgroup of $A$ is
\[
A[p^s]=\{a\in A:a^{p^s}=1\}.
\]
Since an element of a $q$-primary component with $q\ne p$ cannot have nontrivial $p$-power order, one has $A[p^s]=A_p[p^s]$. The cyclic decomposition of $A_p$ further yields
\[
A[p^s]
\cong
\bigoplus_{i=1}^{d_p}
\mathbb Z_{p^{\min\{s,\lambda_{p,i}\}}}.
\]
In particular, $A[p]\cong \mathbb Z_p^{\,d_p}$ is an $\mathbb F_p$-vector space, and $d_p$ is called the \emph{$p$-rank} of $A$.

\subsection{Ramanujan sums over cyclic subgroups}
We briefly recall the Ramanujan sums used below. Their use in subset-sum enumeration over cyclic groups goes back to \cite{Ramanathan1944}. If \(d\ge1\), \(\zeta_d\) is a primitive \(d\)-th root of unity, and \(r\in\mathbb Z\), the classical Ramanujan sum is
\[
c_d(r)=\sum_{a\in(\mathbb Z/d\mathbb Z)^\times}\zeta_d^{ar}.
\]
For a cyclic subgroup $C\le\widehat G$ and $x\in G$, let
\[
\Gen(C)=\{\eta\in C:\langle\eta\rangle=C\},
\qquad
R_C(x):=\sum_{\eta\in\Gen(C)}\eta(-x).
\]
Thus $R_C(x)$ is the corresponding Ramanujan sum over the generators of $C$.

We record several elementary properties of $R_C(x)$; each follows directly from the definition.

\begin{remark}\label{rem:generator-sum-properties}
	The following properties hold.
	\begin{enumerate}[(i)]
		\item If $C$ has prime order $p$, then
		$$
		R_C(x)
		=
		\begin{cases}
			p-1,&x\in C^\perp,\\
			-1,&x\notin C^\perp.
		\end{cases}
		$$
		
		\item If $C$ has order $p^a$, then
		$$
		|R_C(x)|
		\le
		\varphi(p^a)
		=
		p^{a-1}(p-1).
		$$
		
		\item If $C_1,C_2\le\widehat G$ are cyclic subgroups of coprime
		orders, then $C_1C_2=C_1\times C_2$ is cyclic and
		$$
		R_{C_1C_2}(x)
		=
		R_{C_1}(x)R_{C_2}(x).
		$$
	\end{enumerate}
\end{remark}

\subsection{Zeta transforms on finite posets}

We briefly recall the zeta transform and M\"obius inversion on finite
posets; for details, see \cite[Section~3.7]{Stanley2012}.

Let $\mathcal P$ be a finite poset and let $w:\mathcal P\to\mathbb C$.
The \emph{upper zeta transform} of $w$ is the function $F:\mathcal P\to\mathbb C$
defined by
$$
F(a)
=
\sum_{b\ge a}w(b),
\qquad
a\in\mathcal P.
$$
If $\mu_{\mathcal P}$ denotes the M\"obius function of $\mathcal P$, then
M\"obius inversion gives
$$
w(a)
=
\sum_{b\ge a}\mu_{\mathcal P}(a,b)F(b).
$$
Here $w$ is uniquely determined by its upper zeta transform $F$. We shall
use the following two consequences.

\begin{lemma}\label{lem:zeta-inversion}
	Suppose that
	$$
	F(a)=\sum_{b\ge a}w(b)
	$$
	for every $a\in\mathcal P$. If $F\equiv0$, then $w\equiv0$.
\end{lemma}

\begin{proof}
	By M\"obius inversion,
	$$
	w(a)
	=
	\sum_{b\ge a}\mu_{\mathcal P}(a,b)F(b)
	=
	0
	$$
	for every $a\in\mathcal P$.
\end{proof}

\begin{lemma}\label{lem:one-step-zeta-identity}
	Suppose that
	$$
	F(a)=\sum_{b\ge a}w(b)=\kappa
	$$
	for every $a\in\mathcal P$. Fix $a\in\mathcal P$, and let
	$$
	\mathcal R(a)
	\subseteq
	\{r\in\mathcal P:a<r\}
	$$
	be such that every $b>a$ lies above a unique element of
	$\mathcal R(a)$. Then
	$$
	w(a)
	=
	\bigl(1-|\mathcal R(a)|\bigr)\kappa.
	$$
\end{lemma}

\begin{proof}
	By the defining property of $\mathcal R(a)$, each $b>a$ occurs exactly
	once in the following double sum. Hence
	$$
	|\mathcal R(a)|\kappa
	=
	\sum_{r\in\mathcal R(a)}
	\sum_{b\ge r}w(b)
	=
	\sum_{b>a}w(b).
	$$
	Since $\kappa
	=
	w(a)+\sum_{b>a}w(b)$, it follows that
	$$
	w(a)
	=
	\bigl(1-|\mathcal R(a)|\bigr)\kappa.
	$$
\end{proof}

\section{Characterizations of subset-sum designs via character sums}
\label{sec:characters-subset-sums}

In this section, we develop a character-theoretic approach to subset-sum designs. We derive character-sum characterizations for $1$-designs and $2$-designs and obtain the necessary conditions $\gcd(k,\exp(G))>1$ for a $1$-design and $\exp(G)\mid k$ for a $2$-design. 

\subsection{Characterization of subset-sum \texorpdfstring{$1$}{1}-designs}
In the incidence structure \((G,\cB_k^x)\), the \emph{replication number}
of \(g\in G\) is defined as
$$r_{k}^{x}(g):=\#\{B\in\cB_k^x:\ g\in B\}.$$

\begin{lemma}\label{lem:punctured-subset-count}
	For every \(g\in G\), one has $r_k^x(g)=N(G\setminus\{0\},k-1,x-kg)$.
\end{lemma}

\begin{proof}
	Translation by \(-g\) sends a block \(B\in\cB_k^x\) containing \(g\)
	to a \(k\)-subset \(B-g\) containing \(0\), whose sum is \(x-kg\).  Removing
	this \(0\) gives a \((k-1)\)-subset of \(G\setminus\{0\}\) with sum \(x-kg\).
\end{proof}

The replication number can be viewed as a function $r_k^x(\cdot)$ on $G$. Using Fourier analysis on $r_k^x(\cdot)$, we first give a character-theoretic characterization of when $(G,\cB_k^x)$ is a $1$-design.

\begin{theorem}\label{thm:one-design-criterion}
	Let \(2\le k\le n-1\) and \(x\in G\).  For each
	\(\psi\in\widehat G\), put
	\[
	\mathcal R_k(\psi)
	:=
	\{\chi\in\widehat G:\chi^k=\psi\}.
	\]
	Then \((G,\cB_k^x)\) is a \(1\)-design if and only if, for every
	nontrivial character \(\psi\in\widehat G\),
	\begin{equation}\label{eq:one-design-criterion}
		\mathcal R_k(\psi)=\emptyset
		\quad\text{or}\quad
		\sum_{\chi\in\mathcal R_k(\psi)}
		\chi(-x)
		\operatorname{Coef}_{k-1}\left\{
		\frac{(1-(-Z)^{\ord(\chi)})^{n/\ord(\chi)}}{1+Z}\right\}
		=0.
	\end{equation}
\end{theorem}

\begin{proof}
	For \(y\in G\), character orthogonality gives
	\[
	N(G\setminus\{0\},k-1,y)
	=
	\frac1n
	\sum_{\chi\in\widehat G}
	\chi(-y)
	\operatorname{Coef}_{k-1}\left\{
	\prod_{a\in G\setminus\{0\}}(1+\chi(a)Z)\right\}.
	\]
	If \(\ord(\chi)=d\), then Lemma~\ref{lem:product-identity}, applied with
	\(Z\) replaced by \(-Z\), gives
	\[
	\prod_{a\in G\setminus\{0\}}(1+\chi(a)Z)
	=
	\frac{(1-(-Z)^d)^{n/d}}{1+Z}.
	\]
	Hence, by Lemma~\ref{lem:punctured-subset-count},
	\[
	r_k^x(g)
	=
	\frac1n
	\sum_{\chi\in\widehat G}
	\chi(-x)\chi^k(g)
	\operatorname{Coef}_{k-1}\left\{
	\frac{(1-(-Z)^{\ord(\chi)})^{n/\ord(\chi)}}{1+Z}\right\}.
	\]
	Thus, for each \(\psi\in\widehat G\), the Fourier coefficient of
	\(r_k^x\) at \(\psi^{-1}\) is
	\[
	\begin{aligned}
		\widehat r_k^x(\psi^{-1})
		&=
		\sum_{g\in G}r_k^x(g)\psi^{-1}(g) \\
		&=
		\frac1n
		\sum_{\chi\in\widehat G}
		\chi(-x)
		\operatorname{Coef}_{k-1}\left\{
		\frac{(1-(-Z)^{\ord(\chi)})^{n/\ord(\chi)}}{1+Z}\right\}
		\sum_{g\in G}(\chi^k\psi^{-1})(g) \\
		&=
		\begin{cases}
			0,
			& \mathcal R_k(\psi)=\emptyset, \\[4pt]
			\displaystyle
			\sum_{\chi\in\mathcal R_k(\psi)}
			\chi(-x)
			\operatorname{Coef}_{k-1}\left\{
			\frac{(1-(-Z)^{\ord(\chi)})^{n/\ord(\chi)}}{1+Z}\right\},
			& \mathcal R_k(\psi)\ne\emptyset.
		\end{cases}
	\end{aligned}
	\]
	Since the replication number \(r_k^x\) is constant if and only if
	\(\widehat r_k^x(\psi^{-1})=0\) for every nontrivial
	\(\psi\in\widehat G\), the assertion now follows.
\end{proof}
\begin{proposition}\label{prop:coprime-k-exponent}
	Let \(2\le k\le n-1\) and \(x\in G\). If \((G,\cB_k^x)\) is a \(1\)-design, then
	\(\gcd(k,\exp(G))>1\).
\end{proposition}

\begin{proof}
	Put \(e=\exp(G)\), and suppose that \((k,e)=1\).  Then the map \(\chi\mapsto\chi^k\) is an automorphism of \(\widehat G\). For a nontrivial \(\psi\in\widehat G\), let \(\chi\) be the unique character with \(\chi^k=\psi\).  Then Eq.~\eqref{eq:one-design-criterion} gives
	\begin{equation}\label{eq:one-design-coefficient-zero}
		\chi(-x)
		\operatorname{Coef}_{k-1}\left\{
		\frac{(1-(-Z)^{\ord(\chi)})^{n/\ord(\chi)}}{1+Z}\right\}
		=0.
	\end{equation}
	For \(d=\ord(\chi)\), expanding \(1/(1+Z)\) as a formal power series gives
	\begin{equation}\label{eq:punctured-coefficient}
		\operatorname{Coef}_{k-1}\left\{\frac{(1-(-Z)^d)^{n/d}}{1+Z}\right\}
		=
		(-1)^{k-1}\sum_{j=0}^{\lfloor \frac{k-1}{d}\rfloor}(-1)^j\binom {n/d}{j}
		=
		(-1)^{k-1+\lfloor \frac{k-1}{d}\rfloor}\binom{n/d-1}{\lfloor \frac{k-1}{d}\rfloor}.
	\end{equation}
	Hence this coefficient is nonzero, contradicting Eq.~\eqref{eq:one-design-coefficient-zero}.
\end{proof}

Proposition~\ref{prop:coprime-k-exponent} extends to arbitrary finite abelian groups the necessary divisibility condition $p\mid k$ previously obtained for elementary abelian $2$-groups in the zero-sum case \cite[Remark~2.7(1)]{FalconePavone2021}, for elementary abelian $p$-groups with $p$ odd \cite[Theorem~3.1]{Pavone2023}, and, more generally, for finite abelian $p$-groups \cite[Theorem~3.3 and Proposition~3.4]{LiuTangFanLuo2025}.

\begin{remark}\label{rem:p-group-one-design}
	When \(G\) is a finite abelian \(p\)-group, Theorem~\ref{thm:one-design-criterion} has a simple form. Suppose \(\exp(G)=p^\alpha\), and write \(k=p^h u\) with \((u,p)=1\). Then \((G,\cB_k^x)\) is a \(1\)-design if and only if either $h\ge\alpha$ or $1\le h<\alpha$ and $x\notin p^hG$. This recovers \cite[Theorem~3.3]{LiuTangFanLuo2025}.

	Indeed, if \(h\ge\alpha\), then \(\chi^k=1\) for every
	\(\chi\in\widehat G\), so the sums in Eq.~\eqref{eq:one-design-criterion}
	are empty for all nontrivial \(\psi\). Now assume \(h<\alpha\). Since \(\chi\mapsto\chi^u\) is an automorphism of \(\widehat G\), the fibres of \(\chi\mapsto\chi^k\) are, up to this automorphism, the same as those of \(\chi\mapsto\chi^{p^h}\). Every nontrivial fibre is therefore a coset \(\chi_0\widehat G[p^h]\), where
	\(\widehat G[p^h]=\{\tau\in\widehat G:\tau^{p^h}=1\}\). Then \(\ord(\chi_0)>p^h\), and therefore
	\(\ord(\chi_0\tau)=\ord(\chi_0)\) for every
	\(\tau\in\widehat G[p^h]\). The corresponding sum in
	Eq.~\eqref{eq:one-design-criterion} is thus
	\begin{equation}\label{eq:p-group-fibre-sum}
		\chi_0(-x)
		\operatorname{Coef}_{k-1}\left\{
		\frac{(1-(-Z)^{\ord(\chi_0)})^{n/\ord(\chi_0)}}{1+Z}\right\}
		\sum_{\tau\in\widehat G[p^h]}\tau(-x).
	\end{equation}
	By Eq.~\eqref{eq:punctured-coefficient}, the coefficient preceding the character sum is
	nonzero. Moreover, character orthogonality gives
	$$
	\sum_{\tau\in\widehat G[p^h]}\tau(-x)
	=
	\begin{cases}
		|\widehat G[p^h]|, & x\in p^hG,\\
		0, & x\notin p^hG.
	\end{cases}
	$$
	Hence Eq.~\eqref{eq:p-group-fibre-sum} vanishes if and only if \(x\notin p^hG\).
\end{remark}

\subsection{Characterization of subset-sum \texorpdfstring{$2$}{2}-designs}

We begin with the standard incidence-matrix formulation of a
$2$-design; see, for example,
\cite[Chapter~I]{BethJungnickelLenz1999}.
Let $(\mathcal P,\mathcal B)$ be a $2$-$(n,k,\lambda)$ design
with replication number $r$, and let $A$ be its point--block
incidence matrix, with rows indexed by $\mathcal P$ and columns
indexed by $\mathcal B$. The matrix $AA^{\mathsf T}$ is the 
\emph{concurrence matrix} of the design: its $(u,v)$-entry is the number
of blocks containing both $u$ and $v$. Thus one has

$$
AA^{\mathsf T}=(r-\lambda)I+\lambda J,
$$

where $J$ is the all-one matrix.

For a function $f:\mathcal P\to\mathbb C$, fix an ordering
$\mathcal P=\{u_1,\ldots,u_n\}$ and let
$\mathbf f=(f(u_1),\ldots,f(u_n))^{\mathsf T}\in\mathbb C^n$. Denote its conjugate transpose by
$\mathbf f^*=(\overline{f(u_1)},\ldots,\overline{f(u_n)})$.
The $B$-coordinate of $A^{\mathsf T}\mathbf f$ is
$\sum_{u\in B}f(u)$. Hence

\begin{equation}\label{eq:design-second-moment}
	\begin{aligned}
		\sum_{B\in\mathcal B}
		\left|\sum_{u\in B}f(u)\right|^2
		&=\lVert A^{\mathsf T}\mathbf f\rVert^2\\
		&=\mathbf f^*AA^{\mathsf T}\mathbf f\\
		&=(r-\lambda)\sum_{u\in\mathcal P}|f(u)|^2
		+\lambda\left|\sum_{u\in\mathcal P}f(u)\right|^2.
	\end{aligned}
\end{equation}

Indeed, if $\mathbf 1$ denotes the all-one column vector, then
$J=\mathbf 1\mathbf 1^*$ and
$\mathbf f^*J\mathbf f
=\lvert\mathbf 1^*\mathbf f\rvert^2
=\left|\sum_{u\in\mathcal P}f(u)\right|^2$.

We apply this identity to the incidence structure
$(G,\cB_k^x)$. For $\chi\in\widehat G$, define

$$
E_\chi=
\sum_{B\in\cB_k^x}
\left|\sum_{b\in B}\chi(b)\right|^2.
$$

\begin{proposition}\label{prop:design-energy-constant}
	If $(G,\cB_k^x)$ is a $2$-$(n,k,\lambda)$ design, then $E_\chi$
	is independent of the choice of nontrivial character $\chi\ne1$.
	More precisely, we have
	
	$$
	E_\chi=(r-\lambda)n.
	$$
\end{proposition}

\begin{proof}
	Since every character value has modulus one,
	$\sum_{g\in G}|\chi(g)|^2=n$. If $\chi\ne1$, character
	orthogonality gives $\sum_{g\in G}\chi(g)=0$. Applying
	Eq.~\eqref{eq:design-second-moment} with $f=\chi$ gives
	
	$$
	E_\chi=(r-\lambda)n,
	$$

\end{proof}

The following theorem converts the design condition into a structural constraint on the character group. By Proposition~\ref{prop:design-energy-constant}, a $2$-design forces $E_\chi$ to be constant as $\chi$ ranges over the nontrivial characters. The expansion below identifies the $\chi$-dependent part of $E_\chi$ with a sum over the cyclic subgroups $C\le\widehat G$ containing $\langle\chi\rangle$. For each divisor $d$ of $\exp(G)$, set
\begin{equation}\label{eq:Ad-definition}
A_d=
\begin{cases}
n(-1)^{k+k/d-3}
\displaystyle\binom{n/d-2}{k/d-1},
& d\mid k,\\[2mm]
0, & d\nmid k.
\end{cases}
\end{equation}

\begin{theorem}\label{thm:energy-expansion}
	There exists a constant \(\beta\), independent of \(\chi\), such that for every nontrivial \(\chi\in\widehat G\),
	\begin{equation}
	E_\chi
	=
	\beta+
	\sum_{\substack{C\le\widehat G\ \mathrm{cyclic}\\
		\langle\chi\rangle\le C}}
	A_{|C|}R_C(x),
	\end{equation}
	where $R_C(x)$ is the Ramanujan sum defined above. Consequently, if $(G,\cB_k^x)$ is a $2$-design, then
	\begin{equation}\label{eq:energy-system}
	\sum_{\substack{C\le\widehat G\ \mathrm{cyclic}\\C_0\le C}}A_{|C|}R_C(x)
	\end{equation}
	is independent of the choice of nontrivial cyclic subgroup $C_0\le\widehat G$.
\end{theorem}

\begin{proof}
	By the definition of \(E_\chi\) and character orthogonality,
	\[
	\begin{aligned}
	E_\chi
	&=
	\frac1n
	\sum_{\eta\in\widehat G}
	\eta(-x)
	\operatorname{Coef}_{k}\left\{
	\sum_{B\subseteq G}
	\left|\sum_{b\in B}\chi(b)\right|^2
	\eta\left(\sum_{b\in B}b\right)
	Z^{|B|}\right\}.
	\end{aligned}
	\]
Fix \(\eta\in\widehat G\). We rewrite the inner subset sum using the ordered pair \((u,v)\) that appears after expanding the square. Namely,
$\left|\sum_{b\in B}\chi(b)\right|^2=\sum_{u,v\in B}\chi(u)\chi(-v)$.

Hence
\begin{equation}\label{eq:pair-counting-view}
	\begin{aligned}
		&\sum_{B\subseteq G}
		\left|\sum_{b\in B}\chi(b)\right|^2
		\eta\left(\sum_{b\in B}b\right)Z^{|B|} \\
		&\qquad =
		\sum_{u,v\in G}\chi(u)\chi(-v)
		\sum_{\substack{B\subseteq G\\ u,v\in B}}
		\eta\left(\sum_{b\in B}b\right)Z^{|B|}.
	\end{aligned}
\end{equation}

Put $S_\eta(Z)=\prod_{g\in G}(1+\eta(g)Z)$ and $r_g=\frac{\eta(g)Z}{1+\eta(g)Z}$. If \(u\ne v\), then forcing \(u\) and \(v\) to lie in \(B\) gives
\begin{equation}\label{eq:off-diagonal-forced-choice}
	\begin{aligned}
		\sum_{\substack{B\subseteq G\\ u,v\in B}}
		\eta\left(\sum_{b\in B}b\right)Z^{|B|}
		&=
		\eta(u)\eta(v)Z^2
		\prod_{g\in G\setminus\{u,v\}}(1+\eta(g)Z)  \\
		&=
		S_\eta(Z)r_ur_v.
	\end{aligned}
\end{equation}
If \(u=v\), then 
\begin{equation}\label{eq:diagonal-forced-choice}
	\begin{aligned}
		\sum_{\substack{B\subseteq G\\ u\in B}}
		\eta\left(\sum_{b\in B}b\right)Z^{|B|}
		&=
		\eta(u)Z
		\prod_{g\in G\setminus\{u\}}(1+\eta(g)Z)  \\
		&=
		S_\eta(Z)r_u.
	\end{aligned}
\end{equation}
Substituting Eqs.~\eqref{eq:off-diagonal-forced-choice} and~\eqref{eq:diagonal-forced-choice} into Eq.~\eqref{eq:pair-counting-view}, we obtain
\begin{equation}\label{eq:inner-sum-offdiag-diag}
	\begin{aligned}
		&\sum_{B\subseteq G}
		\left|\sum_{b\in B}\chi(b)\right|^2
		\eta\left(\sum_{b\in B}b\right)Z^{|B|} \\
		&\qquad =
		S_\eta(Z)
		\left(\sum_{\substack{u,v\in G\\ u\ne v}}
		\chi(u)\chi(-v)r_ur_v
		+
		\sum_{u\in G}r_u
		\right).
	\end{aligned}
\end{equation}
For the first summand in Eq.~\eqref{eq:inner-sum-offdiag-diag}, we have
\begin{equation}\label{eq:offdiag-completed}
	\begin{aligned}
		\sum_{\substack{u,v\in G\\ u\ne v}}
		\chi(u)\chi(-v)r_ur_v
		&=
		\sum_{u,v\in G}\chi(u)\chi(-v)r_ur_v
		-
		\sum_{g\in G}\chi(g)\chi(-g)r_g^2  \\
		&=
	\sum_{g\in G}\chi(g)r_g
	\sum_{g\in G}\chi(-g)r_g
		-
		\sum_{g\in G}r_g^2.
	\end{aligned}
\end{equation}
Combining Eqs.~\eqref{eq:inner-sum-offdiag-diag} and~\eqref{eq:offdiag-completed}, we get
\begin{equation}\label{eq:inner-sum-final}
	\begin{aligned}
		&\sum_{B\subseteq G}
		\left|\sum_{b\in B}\chi(b)\right|^2
		\eta\left(\sum_{b\in B}b\right)Z^{|B|} \\
		&\qquad =
		S_\eta(Z)
		\left[
		\left(\sum_{g\in G}\chi(g)r_g\right)
		\left(\sum_{g\in G}\chi(-g)r_g\right)
		+
		\sum_{g\in G}(r_g-r_g^2)
		\right].
	\end{aligned}
	\end{equation}

Observe that the second term inside the brackets is independent of \(\chi\). After summing over \(\eta\) and extracting the coefficient of \(Z^k\), all such contributions are absorbed into the constant \(\beta\).

It remains to compute
\[
S_\eta(Z)
\left(\sum_{g\in G}\chi(g)r_g\right)
\left(\sum_{g\in G}\chi(-g)r_g\right).
\]
Expanding \(r_g\) as a power series, we obtain $r_g=\sum_{j\ge1}(-1)^{j-1}\eta(g)^j Z^j$.
Then we have
\[
\sum_{g\in G}\chi(g)r_g
=
\sum_{j\ge1}(-1)^{j-1}Z^j
\sum_{g\in G}(\chi\eta^j)(g).
\]
Thus the inner sum is nonzero precisely when \(\eta^j=\chi^{-1}\).
Similarly,
\[
\sum_{g\in G}\chi(-g)r_g
=
\sum_{j\ge1}(-1)^{j-1}Z^j
\sum_{g\in G}(\chi^{-1}\eta^j)(g),
\]
and here the inner sum is nonzero precisely when \(\eta^j=\chi\).
Hence the product is zero unless \(\chi\in\langle\eta\rangle\).

Assume now that \(\chi\in\langle\eta\rangle\), and let
\(d=\ord(\eta)\). Choose \(1\le a\le d-1\) such that
\(\eta^a=\chi^{-1}\). Then
\[
\sum_{g\in G}\chi(g)r_g
=
n\sum_{t\ge0}(-1)^{a+td-1}Z^{a+td}
=
n\frac{(-1)^{a-1}Z^a}{1-(-1)^d Z^d},
\]
and
\[
\sum_{g\in G}\chi(-g)r_g
=
n\sum_{t\ge0}(-1)^{d-a+td-1}Z^{d-a+td}
=
n\frac{(-1)^{d-a-1}Z^{d-a}}{1-(-1)^d Z^d}.
\]
Their product is therefore
\begin{equation}\label{eq:dependent-energy-term}
\sum_{g\in G}\chi(g)r_g
\sum_{g\in G}\chi(-g)r_g=n^2(-1)^{d-2}\frac{Z^d}{(1-(-1)^d Z^d)^2}.
\end{equation}

By Lemma~\ref{lem:product-identity} and Eq.~\eqref{eq:dependent-energy-term}, we have
\[
S_\eta(Z)
\left(\sum_{g\in G}\chi(g)r_g\right)
\left(\sum_{g\in G}\chi(-g)r_g\right)=n^2(-1)^{d-2}Z^d
\bigl(1-(-1)^d Z^d\bigr)^{n/d-2}.
\]
Therefore the contribution of this \(\eta\) to the \(\chi\)-dependent
part, including the outer factor \(n^{-1}\eta(-x)\), is
\[
\begin{aligned}
	&\frac{1}{n}\eta(-x)\operatorname{Coef}_{k}\left\{S_\eta(Z)
	\left(\sum_{g\in G}\chi(g)r_g\right)
	\left(\sum_{g\in G}\chi(-g)r_g\right)\right\} \\
	&\qquad
	=\eta(-x)\,n(-1)^{d-2}
	\operatorname{Coef}_{k-d}\left\{
	\bigl(1-(-1)^dZ^d\bigr)^{n/d-2}\right\} \\
	&\qquad
	=\eta(-x)A_d.
\end{aligned}
\]
Consequently, the $\chi$-dependent part is obtained by summing $\eta(-x)A_{\ord(\eta)}$ over all $\eta$ with $\chi\in\langle\eta\rangle$. Grouping these characters by the cyclic subgroup $C=\langle\eta\rangle$ gives
\[
\sum_{\substack{C\le\widehat G\ \mathrm{cyclic}\\
		\langle\chi\rangle\le C}}
A_{|C|}R_C(x).
\]
This completes the proof.
\end{proof}

\begin{proposition}\label{prop:exp-divides-k}
	If \((G,\cB_k^x)\) is a \(2\)-design, then \(\exp(G)\mid k\).
\end{proposition}

\begin{proof}
	Put \(e=\exp(G)\).  Since a \(2\)-design is a \(1\)-design,
	Proposition~\ref{prop:coprime-k-exponent} gives \(\gcd(k,e)>1\).  Suppose, for
	contradiction, that \(e\nmid k\).  Choose a prime \(p\mid e\) such that
	\(h=\nu_p(k)<\nu_p(e)\).
	
	Let \(D\le\widehat G\) be a cyclic subgroup of order \(p^{h+1}\).  If \(C\) is
	a cyclic subgroup of \(\widehat G\) containing \(D\), then
	\(p^{h+1}\mid |C|\), and hence \(|C|\nmid k\).  Thus \(A_{|C|}=0\) by
	Eq.~\eqref{eq:Ad-definition}, and so
	\[
	\sum_{\substack{C\le\widehat G\ \mathrm{cyclic}\\D\le C}}
	A_{|C|}R_C(x)=0.
	\]
	By Theorem~\ref{thm:energy-expansion}, the same expression is constant as
	\(D\) ranges over the nontrivial cyclic subgroups of \(\widehat G\).  Hence
	this constant is zero.  Applying Lemma~\ref{lem:zeta-inversion} to the finite
	poset of nontrivial cyclic subgroups of \(\widehat G\), ordered by inclusion,
	we obtain \(A_{|C|}R_C(x)=0\) for every nontrivial cyclic subgroup
	\(C\le\widehat G\).
	
	Choose a prime \(q\mid\gcd(k,e)\), and let \(C\le\widehat G\) be cyclic of order \(q\). Since \(2\le k\le n-1\), the coefficient \(A_q\) is nonzero by Eq.~\eqref{eq:Ad-definition}. On the other hand, \(R_C(x)\ne 0\) by Remark~\ref{rem:generator-sum-properties}~(i). This contradiction proves \(e\mid k\).
\end{proof}

\section{Proof of Theorem~\ref{thm:main}}\label{sec:local-obstruction}

In this section, we use the preceding results to derive structural constraints on $\widehat G$, and hence on $G$. For each prime \(p\mid |\widehat G|\), let \(\widehat G_p\) be the $p$-primary component of $\widehat G$, and put \(\widehat G^{(p)}=\prod_{q\ne p}\widehat G_q\). Thus $\widehat G=\widehat G_p\times\widehat G^{(p)}$. Under the dual primary decomposition of \(G\), write $G=G_p\oplus G^{(p)}$, where $G^{(p)}=\bigoplus_{q\ne p}G_q$, and write $x=x_p+x^{(p)}$. Put $d=\exp(G^{(p)})$.

By Proposition~\ref{prop:exp-divides-k} and Theorem~\ref{thm:energy-expansion}, a hypothetical $2$-design gives $\exp(G)\mid k$ and a complex number $\kappa$ such that
\begin{equation}\label{eq:global-equations-expanded}
F(D)=\sum_{\substack{C\le \widehat G\ \mathrm{cyclic}\\D\le C}}
A_{|C|}R_C(x)
=\kappa,
\end{equation}
for every nontrivial cyclic subgroup \(D\le\widehat G\). We first prove that, under the $2$-design hypothesis, every $p$-primary component of $G$ is elementary abelian.

\begin{theorem}\label{thm:prime-local-obstruction}
	Suppose that $\exp(G)\mid k$ and Eq.~\eqref{eq:global-equations-expanded} holds.  Then
	\(\exp(\widehat G_p)=p\) for every prime \(p\mid \exp(G)\).
\end{theorem}

We shall use the following localized form of Eq.~\eqref{eq:global-equations-expanded}. From now on, write $e=\exp(G)$.

\begin{proposition}\label{prop:sylow-localization}
	Suppose that Eq.~\eqref{eq:global-equations-expanded} holds. For a fixed prime $p\mid\exp(\widehat G)$, let \(M\le\widehat G^{(p)}\) be cyclic of order \(d\). Then \(R_M(x^{(p)})\ne 0\). Moreover, for every nontrivial cyclic subgroup \(D\le\widehat G_p\),
	\begin{equation}\label{eq:local-equations}
	\sum_{\substack{E\le \widehat G_p\ \mathrm{cyclic}\\D\le E}}
	A_{d|E|}R_E(x_p)
	=
	\kappa_p,
	\qquad
	\kappa_p=\frac{\kappa}{R_M(x^{(p)})}\ne 0.
	\end{equation}
\end{proposition}

\begin{proof}
	Choose a cyclic subgroup \(N\le\widehat G_p\) of order \(\exp(\widehat G_p)\). Then \(NM\) is cyclic of order \(e\), hence maximal among the cyclic subgroups of \(\widehat G\). Applying Eq.~\eqref{eq:global-equations-expanded} to \(NM\) gives $\kappa=A_eR_{NM}(x)$. By Remark~\ref{rem:generator-sum-properties}~(iii), \(R_{NM}(x)=R_N(x_p)R_M(x^{(p)})\). We claim that $\kappa\ne0$. Indeed, if $\kappa=0$, Lemma~\ref{lem:zeta-inversion} applied to Eq.~\eqref{eq:global-equations-expanded} would give
	$$
	A_{|C|}R_C(x)=0
	$$
	for every nontrivial cyclic subgroup $C\le\widehat G$. Taking $C$ of prime order $q\mid e$, we have $q\mid e\mid k$, and hence $A_q\ne0$. On the other hand, Remark~\ref{rem:generator-sum-properties}~(i) gives $R_C(x)\ne0$, a contradiction. Thus $\kappa\ne0$.
	
	By Remark~\ref{rem:generator-sum-properties}~(iii),
	$$
	R_{NM}(x)=R_N(x_p)R_M(x^{(p)}).
	$$
	Since $\kappa=A_eR_{NM}(x)\ne0$, it follows that $R_M(x^{(p)})\ne0$.  
	
	Fix a nontrivial cyclic subgroup \(D\le \widehat G_p\).  The cyclic
	subgroups of \(\widehat G\) containing \(DM\) are precisely the subgroups
	\(EM\), where \(E\le \widehat G_p\) is cyclic and \(D\le E\). Using Eq.~\eqref{eq:global-equations-expanded} for \(DM\) and again
	Remark~\ref{rem:generator-sum-properties}~(iii), we obtain
	\[
	\begin{aligned}
		\kappa
		&=
		\sum_{\substack{E\le \widehat G_p\ \mathrm{cyclic}\\D\le E}}
		A_{d|E|}R_{EM}(x)                                      \\
		&=
		R_M(x^{(p)})
		\sum_{\substack{E\le \widehat G_p\ \mathrm{cyclic}\\D\le E}}
		A_{d|E|}R_E(x_p).
	\end{aligned}
	\]
	Dividing by \(R_M(x^{(p)})\) proves Eq.~\eqref{eq:local-equations}.
\end{proof}

From now on, put \(S=\widehat G_p\), and let $S^p=\{\chi^p:\chi\in S\}$. By Section~\ref{sec:preliminaries}, the group \(S[p]\) is an \(\mathbb F_p\)-vector space whose lines are precisely the cyclic subgroups of \(S\) of order \(p\). Write \(|S[p]|=p^r\), and put \(\ell=|S[p]|/p=p^{r-1}\).

\begin{lemma}\label{lem:root-count}
	With the notation above, let \(D\le S\) be a nontrivial cyclic subgroup, and let \(\mathcal R(D)\) be the set of cyclic subgroups \(E\le S\) such that \(D<E\) and \(|E|=p|D|\). Then
	\[
	|\mathcal R(D)|
	=
	\begin{cases}
		0, & \text{if }D\nleq S^p,\\
		\ell, & \text{if }D\le S^p.
	\end{cases}
	\]
\end{lemma}

\begin{proof}
	Let \(|D|=p^s\).  A subgroup in \(\mathcal R(D)\) is generated by an element
	\(y\in S\) with \(\langle y^p\rangle=D\).  If \(D\nleq S^p\), no such \(y\)
	exists.  Assume \(D\le S^p\).  Since the kernel of \(y\mapsto y^p\) is \(S[p]\),
	the number of \(y\in S\) with \(\langle y^p\rangle=D\) is
	\(\varphi(p^s)|S[p]|\).  Each subgroup in \(\mathcal R(D)\) has
	\(\varphi(p^{s+1})=p\varphi(p^s)\) such generators.  Hence
	\[
	|\mathcal R(D)|
	=
	\frac{\varphi(p^s)|S[p]|}{\varphi(p^{s+1})}
	=
	\frac{|S[p]|}{p}
	=
	\ell.
	\]
\end{proof}

Proposition~\ref{prop:sylow-localization} will be used in the following form.

\begin{lemma}\label{lem:local-weight-identity}
	For every nontrivial cyclic subgroup \(D\le S\),
	\begin{equation}\label{eq:local-weight-identity}
		A_{d|D|}R_D(x_p)
		=
		\begin{cases}
			\kappa_p, & \text{if }D\nleq S^p,\\[1mm]
			(1-\ell)\kappa_p, & \text{if }D\le S^p.
		\end{cases}
	\end{equation}
\end{lemma}

\begin{proof}
	Put \(w(E)=A_{d|E|}R_E(x_p)\) for nontrivial cyclic \(E\le S\).  By
	Eq.~\eqref{eq:local-equations}, \(\sum_{E\ge D}w(E)=\kappa_p\) for every
	nontrivial cyclic \(D\le S\).  For fixed \(D\), set
	\[
	\mathcal R(D)=\{E\le S:\ E\text{ is cyclic},\ D<E,\ |E|=p|D|\}.
	\]
	If \(C>D\) is cyclic, then \(C\) has a unique subgroup \(E\) of order
	\(p|D|\), and this \(E\) contains \(D\).  Thus \(\mathcal R(D)\) satisfies the
	hypothesis of Lemma~\ref{lem:one-step-zeta-identity}, and so
	\(w(D)=(1-|\mathcal R(D)|)\kappa_p\).  The formula follows from
	Lemma~\ref{lem:root-count}.
\end{proof}

We prove Theorem~\ref{thm:prime-local-obstruction} by contradiction. Suppose that \(\exp(S)=p^\alpha\), where \(\alpha\ge2\), and let \(\mathcal C\) be the set of cyclic subgroups of $S$ of order $p^\alpha$. Since every \(C\in\mathcal C\) is maximal among the cyclic subgroups of $S$, Eq.~\eqref{eq:local-equations} gives
\begin{equation}\label{eq:top-order-value}
A_eR_C(x_p)=\kappa_p.
\end{equation}

\begin{proposition}\label{prop:top-order-annihilation}
	Suppose that Eq.~\eqref{eq:global-equations-expanded} holds and \(\exp(S)=p^\alpha\), where \(\alpha\ge2\). Let
	\[
	S[p^{\alpha-1}]^\perp=\{y\in G_p:\chi(y)=1\text{ for every }\chi\in S[p^{\alpha-1}]\}.
	\]
	Then \(x_p\in S[p^{\alpha-1}]^\perp\). In particular, \(\eta(x_p)=1\) for every \(\eta\in S[p]\). Consequently, for every subgroup \(L\le S[p]\) of order \(p\),
	\begin{equation}\label{eq:line-generator-sum-new}
	R_L(x_p)=p-1.
	\end{equation}
\end{proposition}

\begin{proof}
	By Eq.~\eqref{eq:top-order-value}, the values \(R_C(x_p)\), \(C\in\mathcal C\),
	are equal and nonzero.  Hence
	\[
	0\ne \sum_{C\in\mathcal C}R_C(x_p)
	=
	\sum_{\substack{\eta\in S\\ \ord(\eta)=p^\alpha}}\eta(-x_p),
	\]

	If \(x_p=0\), the assertion is immediate.  Assume \(x_p\ne 0\).  Character
	orthogonality gives \(\sum_{\eta\in S}\eta(-x_p)=0\).  Since the elements of
	order strictly smaller than \(p^\alpha\) are exactly those in
	\(S[p^{\alpha-1}]\), we get
	\[
	\sum_{\substack{\eta\in S\\ \ord(\eta)=p^\alpha}}\eta(-x_p)
	=
	-\sum_{\eta\in S[p^{\alpha-1}]}\eta(-x_p).
	\]
	The left-hand side is nonzero, so the last sum is nonzero.  By orthogonality
	again, \(x_p\in S[p^{\alpha-1}]^\perp\).
	
	Since \(\alpha\ge2\), we have \(S[p]\le S[p^{\alpha-1}]\).  Thus every
	\(\eta\in S[p]\) satisfies \(\eta(x_p)=1\).  If \(L\le S[p]\) has order
	\(p\), then all generators of \(L\) evaluate trivially at \(x_p\), and
	Eq.~\eqref{eq:line-generator-sum-new} follows.
\end{proof}

\begin{proposition}\label{prop:torsion-contained-in-pS}
	Assume that Eq.~\eqref{eq:global-equations-expanded} holds and \(\exp(S)=p^\alpha\), where \(\alpha\ge2\). Then \(S[p]\subseteq S^p\) and \(\dim_{\mathbb F_p}S[p]\ge2\).
\end{proposition}

\begin{proof}
	For every line \(L\le S[p]\), Proposition~\ref{prop:top-order-annihilation}
	and Eq.~\eqref{eq:local-weight-identity} give
	\begin{equation}\label{eq:line-local-dichotomy}
		A_{dp}(p-1)
		=
		\begin{cases}
			\kappa_p,
			&\text{if }L\nleq S^p,\\[1mm]
			(1-\ell)\kappa_p,
			&\text{if }L\le S^p.
		\end{cases}
	\end{equation}
	Let \(U=S[p]\cap S^p\).  Since \(\alpha\ge2\), we have \(U\ne 0\).  If
	\(U\ne S[p]\), then there are lines \(L_0\le U\) and \(L_1\nleq U\).
	Applying Eq.~\eqref{eq:line-local-dichotomy} to \(L_0\) and \(L_1\) gives
	\((1-\ell)\kappa_p=\kappa_p\), a contradiction.  Hence \(S[p]\subseteq S^p\).
	
	If \(\dim_{\mathbb F_p}S[p]=1\), then \(\ell=1\), and the unique line in
	\(S[p]\) lies in \(S^p\).  Eq.~\eqref{eq:line-local-dichotomy} gives
	\(A_{dp}(p-1)=0\), contrary to Eq.~\eqref{eq:Ad-definition}, since
	\(dp\mid e\mid k\).
\end{proof}

We will use the following elementary estimate.

\begin{lemma}\label{lem:binomial-estimate}
	Let \(a\ge2\), \(\ell\ge2\), \(N\ge \ell^2\), and
	\(1\le m\le N-1\).  Then
	\[
	\binom{aN-2}{am-1}
	>
	a(\ell-1)\binom{N-2}{m-1}.
	\]
\end{lemma}

\begin{proof}
	Vandermonde's identity gives
	\[
	\binom{aN-2}{am-1}
	=
	\sum_j
	\binom{N-2}{j}
	\binom{(a-1)N}{am-1-j}.
	\]
	All terms in the sum are nonnegative, and the term \(j=m-1\) is
	\[
	\binom{N-2}{m-1}
	\binom{(a-1)N}{(a-1)m}.
	\]
	Hence
	\[
	\binom{aN-2}{am-1}
	\ge
	\binom{N-2}{m-1}
	\binom{(a-1)N}{(a-1)m}.
	\]
	Since \(1\le (a-1)m\le (a-1)N-1\), we have
	\(\binom{(a-1)N}{(a-1)m}\ge (a-1)N\).  Finally,
	\[
	(a-1)N\ge (a-1)\ell^2>a(\ell-1),
	\]
	for \(a,\ell\ge2\).  The result follows.
\end{proof}

\begin{proof}[Proof of Theorem~\ref{thm:prime-local-obstruction}]
	It remains to exclude \(\exp(S)=p^\alpha\) with \(\alpha\ge2\). By Proposition~\ref{prop:torsion-contained-in-pS}, \(S[p]\subseteq S^p\) and \(\dim_{\mathbb F_p}S[p]\ge2\). Hence
	\[
	S\cong
	\mathbb Z_{p^{\alpha_1}}\oplus\cdots\oplus\mathbb Z_{p^{\alpha_r}},
	\qquad
	2\le\alpha_1\le\cdots\le\alpha_r=\alpha,
	\]
	with \(r\ge2\). Thus \(\ell=|S[p]|/p=p^{r-1}\). Put
	\(a=p^{\alpha-1}\), \(N=n/e\), and \(m=k/e\). Since
	\(e=dp^\alpha\),
	\begin{equation}\label{eq:N-lower-new}
		N
		=
		\frac{|S|}{p^\alpha}\cdot\frac{|\widehat G^{(p)}|}{d}
		\ge
		p^{2(r-1)}
		=
		\ell^2.
	\end{equation}
	
	Choose a line \(L\le S[p]\).  Since \(S[p]\subseteq S^p\), Eqs.~\eqref{eq:line-generator-sum-new}
	and~\eqref{eq:local-weight-identity} give
	\[
	A_{dp}(p-1)=(1-\ell)\kappa_p.
	\]
	As \(\ell\ge2\), we have
	\begin{equation}\label{eq:kappa-lower-new}
		|\kappa_p|
		=
		\frac{|A_{dp}|(p-1)}{\ell-1}.
	\end{equation}
	
	On the other hand, for any cyclic subgroup \(C\le S\) of order \(p^\alpha\),
	Eq.~\eqref{eq:top-order-value} and Remark~\ref{rem:generator-sum-properties}~(ii)
	give
	\begin{equation}\label{eq:kappa-upper-new}
		|\kappa_p|
		\le
		|A_e|\varphi(p^\alpha)
		=
		|A_e|a(p-1).
	\end{equation}
	
	By Eq.~\eqref{eq:Ad-definition},
	\begin{equation}\label{eq:two-coefficients-new}
		|A_{dp}|
		=
		n\binom{aN-2}{am-1},
		\qquad
		|A_e|
		=
		n\binom{N-2}{m-1}.
	\end{equation}
	Lemma~\ref{lem:binomial-estimate}, together with Eq.~\eqref{eq:N-lower-new},
	yields
	\[
	\binom{aN-2}{am-1}
	>
	a(\ell-1)\binom{N-2}{m-1}.
	\]
	Thus Eq.~\eqref{eq:two-coefficients-new} gives
	\(|A_{dp}|>a(\ell-1)|A_e|\).  This is incompatible with
	Eqs.~\eqref{eq:kappa-lower-new} and~\eqref{eq:kappa-upper-new}.  Therefore
	\(\exp(S)=p\).  Since \(p\mid e\) was arbitrary, the theorem follows.
\end{proof}

We next show that only one primary component can occur under the $2$-design hypothesis. Suppose, to the contrary, that at least two components occur. Write
\[
\widehat G=\widehat G_{p_1}\times\cdots\times\widehat G_{p_s},
\]
where \(s\ge2\), the \(p_i\) are distinct, and \(\widehat G_{p_i}\cong\mathbb Z_{p_i}^{\,a_i}\). Then \(e=\exp(G)=p_1\cdots p_s\). Under the dual decomposition of \(G\), write \(G=G_{p_1}\oplus\cdots\oplus G_{p_s}\) and \(x=x_{p_1}+\cdots+x_{p_s}\), with \(x_{p_i}\in G_{p_i}\).

For each \(i\), let \(\cL_i\) be the set of subgroups of
\(\widehat G_{p_i}\) of order \(p_i\).  For \(L\in\cL_i\),
Remark~\ref{rem:generator-sum-properties}~(i) gives
\begin{equation}\label{eq:order-p-ramanujan}
	R_L(x_{p_i})=
	\begin{cases}
		p_i-1,
		&\text{if }\eta(x_{p_i})=1\text{ for every }\eta\in L,\\
		-1,
		&\text{otherwise}.
	\end{cases}
\end{equation}
In particular, \(R_L(x_{p_i})\ne 0\).

\begin{lemma}\label{lem:mixed-maximal-products}
	With the notation above, assume that \(e\mid k\) and
	Eq.~\eqref{eq:global-equations-expanded} holds.  Then, for each \(i\),
	the value \(R_L(x_{p_i})\) is independent of \(L\in\cL_i\).
\end{lemma}

\begin{proof}
	Choose \(L_i\in\cL_i\).  The subgroup
	\(C=L_1\times\cdots\times L_s\) is cyclic of order \(e\), and hence is
	maximal among cyclic subgroups of \(\widehat G\).  Thus
	Eq.~\eqref{eq:global-equations-expanded} and
	Remark~\ref{rem:generator-sum-properties}~(iii) give
	\begin{equation}\label{eq:max-product-constant}
		A_e\prod_{i=1}^s R_{L_i}(x_{p_i})=\kappa.
	\end{equation}
	Since \(e\mid k\), Eq.~\eqref{eq:Ad-definition} gives \(A_e\ne 0\).
	Fixing all \(L_i\) except \(L_j\), Eq.~\eqref{eq:max-product-constant}
	and Eq.~\eqref{eq:order-p-ramanujan} show that \(R_L(x_{p_j})\) is
	constant on \(\cL_j\).  
\end{proof}

We can now prove the main theorem.

\begin{proof}[Proof of Theorem~\ref{thm:main}]
	Assume that \((G,\cB_k^x)\) is a \(2\)-design.  By
	Proposition~\ref{prop:exp-divides-k}, \(e=\exp(G)\) divides \(k\), and
	Theorem~\ref{thm:energy-expansion} gives
	Eq.~\eqref{eq:global-equations-expanded}.  By
	Theorem~\ref{thm:prime-local-obstruction}, every Sylow component of
	\(\widehat G\) is elementary abelian.  If only one prime divides \(|G|\),
	the conclusion follows.  It remains to exclude the case considered above.
	
	Choose \(j\) with \(p_j\) odd, and fix \(L_i\in\cL_i\) for \(i\ne j\).
	Put \(D=\prod_{i\ne j}L_i\).  The cyclic subgroups of \(\widehat G\)
	containing \(D\) are precisely \(D\) and \(D\times L\), where
	\(L\in\cL_j\).  By Lemma~\ref{lem:mixed-maximal-products}, write
	\(r=R_L(x_{p_j})\), independent of \(L\in\cL_j\).  Applying
	Eq.~\eqref{eq:global-equations-expanded} to \(D\), and comparing with
	Eq.~\eqref{eq:max-product-constant}, gives
	\begin{equation}\label{eq:one-prime-comparison}
		A_{e/p_j}+A_e(|\cL_j|-1)r=0.
	\end{equation}
	
	If \(a_j=1\), then \(|\cL_j|=1\), and
	Eq.~\eqref{eq:one-prime-comparison} gives \(A_{e/p_j}=0\), impossible
	since \(e/p_j\mid e\mid k\).  Hence \(a_j\ge2\).
	
	If \(x_{p_j}\ne 0\), then
	\(\{\eta\in\widehat G_{p_j}:\eta(x_{p_j})=1\}\) is a proper subgroup of
	index \(p_j\).  Since \(a_j\ge2\), it contains some, but not all,
	subgroups of order \(p_j\).  Eq.~\eqref{eq:order-p-ramanujan} would then
	give two different values for \(R_L(x_{p_j})\), contradicting the
	definition of \(r\).  Thus \(x_{p_j}=0\), and \(r=p_j-1\).
	
	Since \(|\cL_j|=(p_j^{a_j}-1)/(p_j-1)\),
	Eq.~\eqref{eq:one-prime-comparison} becomes
	\begin{equation}\label{eq:final-one-prime-contradiction}
		A_{e/p_j}+A_e(p_j^{a_j}-p_j)=0.
	\end{equation}
	Both coefficients are nonzero.  Writing \(k=me\), the difference of their
	sign exponents is
	\[
	\left(k+\frac{k}{e/p_j}-3\right)
	-
	\left(k+\frac{k}{e}-3\right)
	=
	m(p_j-1),
	\]
	which is even.  Hence \(A_{e/p_j}\) and \(A_e\) have the same sign.
	Since \(p_j^{a_j}-p_j>0\), Eq.~\eqref{eq:final-one-prime-contradiction}
	is impossible.
	
	Thus only one prime divides \(|G|\).  Since the corresponding Sylow
	component of \(\widehat G\) has prime exponent, \(\widehat G\), and hence
	\(G\), is an elementary abelian \(p\)-group.
\end{proof}

\end{document}